\documentclass[12pt,a4paper]{amsart}
\usepackage{graphicx}
\usepackage{amsmath,amssymb,amsthm,amsfonts,mathrsfs,enumerate,sansmath,mathtools}

\newtheorem{proposition}{Proposition}[section]
\newtheorem{theorem}{Theorem}[section]
\newtheorem{remark}{Remark}[section]

\usepackage{tikz-cd}
\newcommand{\n}{\noindent}
\usepackage{color}
\allowdisplaybreaks
\begin{document}
	\title[Automorphisms of semidirect products]{Automorphisms of semidirect products fixing the non-normal subgroup}
	\author{Ratan Lal and Vipul Kakkar}
	\address{Department of Mathematics, Central University of Rajasthan, Rajasthan}
	\email{vermarattan789@gmail.com, vplkakkar@gmail.com}
	\date{}
	
		\begin{abstract}
	\n 	In this paper, we describe the automorphism group of semidirect product of two groups that fixes the non-normal subgroup of it. We have computed these automorphisms for the non-abelian metacyclic $p$-group and non-abelian $p$-groups $(p\ge 5)$ of order $p^{4}$, where $p$ is a prime.
	\end{abstract}
	\maketitle

	\textbf{Keywords:} {Automorphism Group; Semidirect Product; $p$-group.}
	
	\textbf{2020 Mathematics Subject Classification:} 20D45, 20E36
	
	\section{Introduction}
	\n	Let $K$ be a non-normal subgroup of a group $G$. Let $S$ be a right transversal to $K$ in $G$ with $1\in S$. Then the group operation on $G$ induces a binary operation on $S$ with respect to it $S$ becomes a right loop. Let $Aut_K(G)=\{\theta\in Aut(G)\mid \theta(K)=K\}$. In \cite[Theorem 2.6, p. 73]{rl}, R. Lal obtained that $\theta \in Aut_K(G)$ can be identified with the triple $(\alpha, \gamma, \delta)$, where $\alpha \in Map(S, S)$, $\gamma \in Map(S, K)$ and $\delta \in Aut(K)$ satisfying the conditions in \cite[Definition 2.5, p. 73]{rl}.
	
	\vspace{0.2 cm}
	
\n In case, if there is a right transversal $H$ to $K$ in $G$ which is a normal subgroup of $G$, then $G$ is the semidirect product of $K$ and $H$. In this case, the conditions on $\alpha, \gamma$ and $\delta$ agrees with the conditions given in \cite[Lemma 1.1, p. 1000]{nch}. These conditions are given as follows.

	\begin{itemize}
		\item[($C1$)] $\alpha(hh^{\prime}) = \alpha(h)\alpha(h^{\prime})^{\gamma(h)}$,
		\item[($C2$)] $\gamma(h^{k}) = \gamma(h)^{\delta(k)}$,
		\item[$(C3)$] $\alpha(h^{k}) = \alpha(h)^{\delta(k)}$, 
		\item[($C4$)] For any $h^{\prime}k^{\prime}\in G$, there exists a unique $hk\in G$ such that $\alpha(h)=h^{\prime}$ and $\gamma(h)\delta(k) = k^{\prime}$.
	\end{itemize} 
	
	\begin{remark}
	In \cite{rl}, the author put the non-normal subgroup in the left in the factorization of $G$. To match the terminology with that in \cite{crn2008}, we put the non-normal subgroup $K$ in the right, that is $G=HK$. Through out the paper, we will use the terminology used in \cite{crn2008}. We will identify the internal semidirect product $G=HK$ with the external semidirect product $H \rtimes_{\phi} K$, where $\phi:K\rightarrow Aut(H)$ is a homomorphism.
	\end{remark}
	
	\n M . J. Curran \cite{crn2008} and D. Jill \cite{dj2005} studied the automorphism group of the group $G$ that fixes the normal subgroup $H$. In this paper, we study the automorphism group of the group $G$ that fixes the non-normal subgroup $K$. We have computed the automorphism group, $Aut_{K}(G)$ in case of non-abelian metacyclic $p$-groups and $p$-groups $(p\ge 5)$ of order $p^{4}$, where $p$ is a prime.

%%%%%%%%%%%%%%%%%%%%%%%%%%%%%%%%%%%%%%%%%%%%%%%%%%%%%%%%%%%%%%%%%%%%%%%%%%%%%%%
	
\section{Structure of the automorphism group $Aut_{K}(G)$} 

Let $	\alpha \in Map(H, H),  \gamma \in Hom(H, K) $ and $\delta \in Aut(K)$. Consider a set	
	\[\mathcal{\hat{M}}_K = \left\{	\begin{pmatrix}
	\alpha & 0\\
	\gamma & \delta
	\end{pmatrix} |\alpha,\gamma, \delta \text{ satisfy the conditions } (C1)-(C4)\right\}.\]
	
	\n Then $\mathcal{\hat{M}}_K$ is a group with the usual multiplication of matrices defined as,
	\begin{equation*}
	\begin{pmatrix}
	\alpha & 0\\
	\gamma & \delta
	\end{pmatrix} \begin{pmatrix}
	\alpha^{\prime} &  0\\
	\gamma^{\prime} &  \delta^{\prime}
	\end{pmatrix} = \begin{pmatrix}
	\alpha \alpha^{\prime} &  0\\
	\gamma\alpha^{\prime} + \delta\gamma^{\prime} &  \delta \delta^{\prime}
	\end{pmatrix},
	\end{equation*}
	where $\alpha \alpha^{\prime} , \delta \delta^{\prime}$ are the usual composition of maps and $(\gamma\alpha^{\prime} + \delta\gamma^{\prime})(h) = \gamma\alpha^{\prime}(h)\delta\gamma^{\prime}(h)$. Clearly, $Aut_{K}(G)$ is a subgroup of $Aut(G)$.   
	\begin{proposition}
		The group $Aut_{K}(G)$ is isomorphic to the group $\mathcal{\hat{M}}_K$.	
	\end{proposition}
	\begin{proof}
	Define a map $\psi: Aut_{K}(G)\longrightarrow \mathcal{\hat{M}}_K$ by $\psi(\theta) = \begin{pmatrix}
	\alpha &  0\\
	\gamma &  \delta
	\end{pmatrix}$, for all $\theta\in Aut_{K}(G)$, where $\alpha: H\longrightarrow H, \gamma: H \longrightarrow K$ and $\delta: K\longrightarrow K$ are defined as, $\theta(h) = \alpha(h)\gamma(h)$ and $\delta(k)=\theta(k)|_K$ for all $h\in H$ and $k\in K$. By the similar arguement as in \cite[Theorem 1, p. 206]{crn2008}, one can observe that $\psi$ is an isomorphism.
	\end{proof}
	
	\n From now on we will identify automorphisms in $Aut_{K}(G)$ with the matrices in $\mathcal{\hat{M}}_K$. Now, we have the following remark.
	
	\begin{remark}
		\begin{itemize}
			\item[$(i)$] $\begin{pmatrix}
			\alpha &  0\\
			0 & 1
			\end{pmatrix}\in Aut_{K}(G)$ if and only if $\alpha\in Aut(H)$ and $\alpha(h^{k}) = {\alpha(h)}^{k}$ for all $h\in H$ and $k\in K$. 
			
			\item[($ii$)] $\begin{pmatrix}
			1 &  0\\
			\gamma &  1
			\end{pmatrix}\in Aut_{K}(G)$ if and only if $\gamma(H)\subseteq C_{K}(H)$ and $\gamma(h^{k}) = {\gamma(h)}^{k}$, for all $h\in H$ and $k\in K$, where $C_{K}(H) = \{k\in K \;| \;h^{k} = h,\; \forall\; h\in H\}$ is the centralizer of $H$ in $K$.
			
			\item[($iii$)] $\begin{pmatrix}
			1 & 0\\
			0 & \delta
			\end{pmatrix}\in Aut_{K}(G)$ if and only if $k^{-1}\delta(k)\in C_{K}(H)$ for all $k\in K$. 
			
		\end{itemize}
	\end{remark}
	\n Now, let us consider the following subsets of $Aut(H), Aut(K)$ and $Aut(H)\times Aut(K)$,  
	\begin{eqnarray*}
		U &=& \{\alpha\in Aut(H)\; |\; \alpha(h^{k}) = {\alpha(h)}^{k}, \forall h\in H, k\in K \},\\
		V &=& \{ \delta \in Aut(K)\; |\; k^{-1}\delta(k)\in C_{K}(H), \forall k\in K\},\\
		W &=& \{(\alpha, \delta)\in Aut(H)\times Aut(K)\; |\; \alpha(h^{k}) = {\alpha(h)}^{\delta(k)}, \forall h\in H, k\in K\}.
	\end{eqnarray*}
	
	\n	Clearly, $U, V$ and $W$ are the subgroups of $Aut(H), Aut(K)$, and $Aut(H) \times Aut(K)$, respectively. The corresponding subgroups of the group $Aut_{K}(G)$ are
	
	$A = \left\{\begin{pmatrix}
	\alpha & 0\\
	0 & 1
	\end{pmatrix} |\; \alpha \in U\right\}, D = \left\{\begin{pmatrix}
	1 & 0\\
	0 & \delta
	\end{pmatrix} |\; \delta \in V\right\}$,\\ and $E = \left\{\begin{pmatrix}
	\alpha & 0\\
	0 & \delta
	\end{pmatrix} |\; (\alpha, \delta)\in W\right\}$. Note that, if $\alpha\in U$ and $\delta\in V$, then $(\alpha, \delta)\in W$. Therefore, $U\times V \le W$.
	
	\vspace{.2cm} %On the other hand, we have 
%	\begin{theorem}\cite[Theorem 2, p. 207]{crn2008}\label{t1}
%		If $U = Aut(H)$ or $V = Aut(K)$, then $W = U \times V$, that is, $E = A\times D$.
%	\end{theorem} 

	\n Clearly, $E$ is a subgroup of $Aut_{K}(G)$. However, one can check that $E$ need not be a normal subgroup of $Aut_{K}(G)$. Let 
	\begin{multline*}
	C = \left\{\begin{pmatrix}
	1 & 0\\
	\gamma & 1
	\end{pmatrix}\in Aut_{K}(G) \;|\; \gamma(H)\subseteq C_{K}(H) \;\text{and}\right.\\ \left. \gamma(h^{k}) = {\gamma(h)}^{k}, \;\text{for all}\; h\in H \;\text{and}\; k\in K\right\}.
	\end{multline*} 
	One can easily observe that $C$ is a normal subgroup of the group $Aut_{K}(G)$. Clearly, $E\cap C = 1$. Now, let $\begin{pmatrix}
	\alpha & 0\\
	\gamma & \delta
	\end{pmatrix}\in Aut_{K}(G)$. 
	Then,
	\begin{equation*}
	\begin{pmatrix}
	\alpha & 0\\
	\gamma & \delta
	\end{pmatrix} = \begin{pmatrix}
	1 & 0\\
	\gamma\alpha^{-1} & 1
	\end{pmatrix} \begin{pmatrix}
	\alpha & 0\\
	0 & \delta
	\end{pmatrix}.
	\end{equation*}	
	Hence, $Aut_{K}(G) = CE, C\trianglelefteq Aut_{K}(G), E\le Aut_{K}(G)$, and $C\cap E = 1$. Thus, we have proved the following theorem,
\begin{theorem}\label{t2}
	Let $G = H \rtimes K$ be the semidirect product. Then $Aut_{K}(G) \simeq C \rtimes E$.
\end{theorem}

%%%%%%%%%%%%%%%%%%%%%%%%%%%%%%%%%%%%%%%%%%%%%%%%%%%%%%%
	\section{Computation of $Aut_K(G)$ for some groups}
		
\n	In this section, we will compute the automorphism group $Aut_{K}(G)$ for non-abelian metacyclic $p$-groups and $p$-groups $(p\ge 5)$ of order $p^{4}$, where $p$ is a prime. $\mathbb{Z}_m$ will denote the cyclic group of order $m$.

	\begin{center}
	\textbf{Metacyclic $p$-groups} 
	\end{center}
	First, assume that $p$ is odd. A non-abelian split metacyclic $p$-group $G$ is of the form $G= \langle a,b \;|\; a^{p^{m}} = 1 = b^{p^{n}}, a^{b} = a^{1+p^{m-r}}\rangle$, where $m\ge 2, n \ge 1$, and $1\le r \le \text{min}\{m-1, n\}$. Let  $H = \langle a\rangle$, $K = \langle b \rangle$ and $\phi: K\longrightarrow Aut(H)$ be defined by $\phi(b)(a) = a^{1+p^{m-r}}$. Then $G=H\rtimes_\phi K$.  
			
			\vspace{.2cm}
			
	\n Note that $[H,K] = \langle a^{p^{m-r}}\rangle \simeq \mathbb{Z}_{p^{r}}$. Since $K$ is abelian, by \cite[Corollary 2.2, p. 490]{bdw2006}, $\gamma(h^k)=\gamma(h)$ is equivalent to $\gamma \in Hom(H/[H,K], K)$. Define $\gamma_{i}:H\rightarrow K$ by $\gamma_{i}(a) = b^{i}, 1\le i \le p^{n}$ when $m-r \ge n$ and by $\gamma_{i}(a) = b^{ip^{n-m+r}} , 1\le i\le p^{m-r}$ when $m-r < n$. Since $[H,K]\subseteq Ker \gamma_i$, it will induce a homomorphism from $H/[H,K]$ to $K$. Let $c =\begin{pmatrix}
	1 & 0\\
	\gamma_{1} & 1
	\end{pmatrix}$. Then, one can easily observe that $\gamma_{1}(H)\subseteq C_{K}(H)$. Therefore, $Hom(H/[H,K], K) \simeq C = \langle c\rangle \simeq \mathbb{Z}_{p^{\text{min}\{m-r,n\}}}$. Also, $C_{K}(H) = \langle b^{p^{r}}\rangle\simeq \mathbb{Z}_{p^{n-r}}$ and for $b\in K$, $b^{-1}\delta(b)\in C_{K}(H)$. Therefore, there are $p^{n-r}$ choices for $\delta(b)$. If $\delta_{1}(b) = b^{1+p^{r}}$, then $V = \langle \delta_{1} \rangle \simeq \mathbb{Z}_{p^{n-r}}$ and so, $D\simeq \mathbb{Z}_{p^{n-r}}$. Now, for all $\alpha \in Aut(H)$, $\alpha(a^{b}) = \alpha(a^{1+p^{m-r}}) = \alpha(a)^{1+p^{m-r}} = \alpha(a)^{b}$. Therefore, $U = Aut(H)\simeq \mathbb{Z}_{p^{m-1}(p-1)}$ and so, $A\simeq \mathbb{Z}_{p^{m-1}(p-1)}$. Then, by Theorem \cite[Theorem 2, p. 207]{crn2008}, $E = A\times D$. Now, by Theorem \ref{t2}, $Aut_{K}(G)\simeq \mathbb{Z}_{p^{\text{min}\{m-r,n\}}}\rtimes (\mathbb{Z}_{p^{m-1}(p-1)} \times \mathbb{Z}_{p^{n-r}})$. Hence, $Aut_{K}(G)$ is a subgroup of index $p^{\text{min}\{m,n\}}$ in the group $Aut(G)$. 
		
		\vspace{.2cm}
		
		\n 	Now, assume $p = 2$. Then, as given in \cite{crn2007}, the non-abelian split metacyclic 2-group is one of the following three forms,
		
		\begin{itemize}
			\item[($i$)] $G = \langle a,b \;|\; a^{2^{m}} = 1 = b^{2^{n}}, a^{b} = a^{1 + 2^{m-r}}\rangle, 1\le r \le \text{min}\{m-2, n\}, m\ge 3, n\ge 1$.
			\item[($ii$)]  $G = \langle a,b \;|\; a^{2^{m}} = 1 = b^{2^{n}}, a^{b} = a^{-1 + 2^{m-r}}\rangle, 1\le r \le \text{min}\{m-2, n\}, m\ge 3, n\ge 1$.
			\item[($iii$)] $G = \langle a,b \;|\; a^{2^{m}} = 1 = b^{2^{n}}, a^{b} = a^{-1}\rangle, m\ge 2, n\ge 1$.
		\end{itemize}
		
		\n Let $H = \langle a\rangle \simeq \mathbb{Z}_{2^{m}}$ and $K = \langle b \rangle \simeq \mathbb{Z}_{2^{n}}$. We will compute the automorphism group, $Aut_{K}(G)$ in the above three cases $(i) - (iii)$. Using the similar argument as for odd prime $p$ above, in the $case(i), [H,K] = \langle a^{2^{m-r}}\rangle \simeq \mathbb{Z}_{2^{r}}$ and $C_{K}(H) = \langle b^{2^{r}}\rangle \simeq \mathbb{Z}_{2^{n-r}}$. Then $Hom(H/[H,K], K) \simeq \mathbb{Z}_{2^{\text{min}\{m-r, n\}}}$. Thus $C \simeq \mathbb{Z}_{2^{\text{min}\{m-r, n\}}}$, $A \simeq \mathbb{Z}_{2}\times \mathbb{Z}_{2^{m-2}}$ and $D \simeq \mathbb{Z}_{2}\times \mathbb{Z}_{2^{n-r-1}}$. Hence, $Aut_{K}(G) \simeq \mathbb{Z}_{2^{\text{min}\{m-r, n\}}} \rtimes (\mathbb{Z}_{2}\times \mathbb{Z}_{2^{m-2}}\times \mathbb{Z}_{2}\times \mathbb{Z}_{2^{n-r-1}})$.  
		
		\vspace{.2cm}
		
		\n		In the case $(ii)$, $[H,K] = \langle a^{2}\rangle \simeq \mathbb{Z}_{2^{m-1}}$ and $C_{K}(H) = \langle b^{2^{r}}\rangle \simeq \mathbb{Z}_{2^{n-r}}$. Thus, $C\simeq \mathbb{Z}_{2}, A \simeq \mathbb{Z}_{2}\times \mathbb{Z}_{2^{m-2}}$ and $D \simeq \mathbb{Z}_{2}\times \mathbb{Z}_{2^{n-r-1}}$. Hence, $Aut_{K}(G) \simeq \mathbb{Z}_{2} \times (\mathbb{Z}_{2}\times \mathbb{Z}_{2^{m-2}}\times \mathbb{Z}_{2}\times \mathbb{Z}_{2^{n-r-1}})$. Similarly, in the $case (iii)$, $[H,K] = \langle a^{2}\rangle \simeq \mathbb{Z}_{2^{m-r}}$, and $C_{K}(H) = \langle b^{2}\rangle \simeq \mathbb{Z}_{2^{n-1}}$. Thus, $C\simeq \mathbb{Z}_{2}, A\simeq \mathbb{Z}_{2}\times \mathbb{Z}_{2^{m-2}}$ and $D \simeq \mathbb{Z}_{2}\times \mathbb{Z}_{2^{n-r-1}}$. Hence, $Aut_{K}(G) \simeq \mathbb{Z}_{2} \times (\mathbb{Z}_{2}\times \mathbb{Z}_{2^{m-2}}\times \mathbb{Z}_{2}\times \mathbb{Z}_{2^{n-2}})$.

\vspace{.2cm}

\begin{center}
\textbf{Non-abelian $p$-groups of order $p^4$ ($p\geq 5$)}
\end{center}

\n Burnside in \cite{wb} classified $p$-groups of order $p^{4}$, where $p$ is a prime. Below, we list $10$ non-abelian $p$-groups ($p\geq 5$) of order $p^4$ up to isomorphism.

\begin{enumerate}
	\item[(i)] $G_1 = \langle a,b\; |\; a^{p^{3}} = 1 = b^{p}, a^{b} = a^{1+p^{2}} \rangle$,
	\item[(ii)] $G_2= \langle a,b \;|\; a^{p^{2}} = 1 = b^{p^{2}}, a^{b} = a^{1+p} \rangle$,
	\item[(iii)] $G_3 = \langle a,b,c\; |\; a^{p^{2}} = 1 = b^{p} = c^{p}, cb = a^{p}bc, ab = ba, ac = ca\rangle$,
	\item[(iv)] $G_4 = \langle a,b,c \;|\; a^{p^{2}} = 1 = b^{p} = c^{p}, ca = a^{1+p}c, ab = ba, cb = bc\rangle$,
	\item[(v)] $G_5 = \langle a,b,c \;|\; a^{p^{2}} = 1 = b^{p} = c^{p}, ca = abc, ab = ba, bc = cb\rangle$,
	\item[(vi)] $G_6 = \langle a,b,c \;|\; a^{p^{2}} = 1 = b^{p} = c^{p}, ba = a^{1+p}b, ca = abc, bc = cb\rangle$,
	\item[(vii)] $G_7 = \langle a,b,c \;|\; a^{p^{2}} = 1 = b^{p} = c^{p}, ba = a^{1+p}b, ca = a^{1+p}bc, cb = a^{p}bc\rangle$,
	\item[(viii)] $G_8 = \langle a,b,c \;|\; a^{p^{2}} = 1 = b^{p} = c^{p}, ba = a^{1+p}b, ca = a^{1+dp}bc, cb = a^{dp}bc, d \not\equiv 0,1 \;(\mod\; p)\rangle$,
	\item[(ix)] $G_9 = \langle a,b,c,d \;|\; a^{p}= b^{p}= c^{p} = d^{p} =1, dc = acd, bd = db, ad = da, bc = cb, ac = ca, ab = ba\rangle$,
	\item[(x)] $G_{10} = \langle a,b,c,d \;|\; a^{p}= b^{p}= c^{p} = d^{p} =1, dc = bcd, db = abd, ad = da, bc = cb, ac = ca, ab = ba\rangle$.
\end{enumerate}

\n Observe that $G_1$ and $G_2$ are metacyclic  $p$-groups. $Aut_K(G_1)$ and $Aut_K(G_2)$ (for the corresponding $K$) can be calculated by the above discussion.

\vspace{0.2 cm}

\n \textbf{The group $G_3$}.  Let $H = \langle a,b \;|\; a^{p^{2}} = b^{p} = 1, ab = ba\rangle$ and $K=\langle c \;|\; c^{p} = 1\rangle$. Then $G_3 \simeq H\rtimes_{\phi} K$, where $\phi: K \longrightarrow Aut(H)$ is given by $\phi(c)(a) = a$ and $\phi(c)(b) = a^{p}b$. Note that $[a^{u}b^{v},c] = (a^{u}b^{v})c(a^{u}b^{v})^{-1}c^{-1} = a^{u}b^{v}(a^{u+pv}b^{v})^{-1} = a^{-pv}$. Therefore, $[H,K] = \langle a^{p}\rangle \simeq \mathbb{Z}_{p}$. Also, if $c^{s}\in C_{K}(H)$, then $a^{i}b^{j} = c^{s}a^{i}b^{j}c^{-s} = a^{i+pjs}b^{j}$. Therefore, $js\equiv 0 \;(\mod\;p)$ for all $j$ and hence, $C_{K}(H) = \{1\}$. This implies that $Hom(H/[H,K],K)$ is the trivial group. Since $K$ is abelian, by \cite[Corollary 2.2, p. 490]{bdw2006} $C$ is the trivial group.
Note that, each $\alpha\in Aut(H)$ defined by $\alpha(a) = a^{i}b^{j}$ and $\alpha(b) = a^{pm}b^{l}$  can be expressed as a matrix $\begin{pmatrix}
i & j\\
m & l
\end{pmatrix}$, where $0\leq i \leq p^2-1, \gcd(p,i) = 1$, $0\le m,j \le p-1$ and $1\leq l \leq p-1$. Also, let $\delta\in Aut(K)\simeq \mathbb{Z}_{p-1}$ be defined by $\delta(c) = c^{r}$, where $1\le r\le p-1$. Now, if $(\alpha, \delta)\in W$, then $(i)$ $\alpha(a^{c}) = \alpha(a)^{\delta(c)}$ and $(ii)$ $\alpha(b^{c}) = \alpha(b)^{\delta(c)}$. By $(i)$, $a^{i}b^{j} = \alpha(a) = \alpha(a^{c}) = \alpha(a)^{\delta(c)} = (a^{i}b^{j})^{c^{r}} = a^{i}a^{prj}b^{j} = a^{i + prj}b^{j}$. Thus, $prj \equiv 0 \;(\mod\;p^{2})$ which implies that $j = 0$. Now, by $(ii)$, $a^{pi+pm}b^{l} = \alpha(a^{p}b) = \alpha(b^{c}) = \alpha(b)^{\delta(c)} = (a^{pm}b^{l})^{c^{r}} = a^{pm}(b^{l})^{c^{r}} = a^{pm}a^{prl}b^{l} = a^{pm+prl}b^{l}$. Thus, $i \equiv rl \;(\mod\; p)$. Let $t$ be a primitive root of $1 \;(\mod\;p)$ and $x = \left( \begin{pmatrix}
t+p & 0\\
0 & t
\end{pmatrix}, \delta_{1}\right), y = \left( \begin{pmatrix}
1 & 0\\
1 & 1
\end{pmatrix}, \delta_{1}\right)$ and $z = \left( \begin{pmatrix}
t+p & 0\\
0 & 1
\end{pmatrix}, \delta_{t}\right)$, where $\delta_{\rho}(c) = c^{\rho}$. Then $W\simeq \langle x,y,z\; |\; x^{p(p-1)} = 1= y^{p}= z^{p(p-1)}, xz=zx, xy = yx, zyz^{-1} = y^{t^{-1}}\rangle$. Therefore, $W \simeq (\mathbb{Z}_{p}\times\mathbb{Z}_{p(p-1)})\rtimes \mathbb{Z}_{p(p-1)}$ and so, $E\simeq (\mathbb{Z}_{p}\times\mathbb{Z}_{p(p-1)})\rtimes \mathbb{Z}_{p(p-1)}$. Hence, by Theorem \ref{t2}, $Aut_{K}(G_3) \simeq  (\mathbb{Z}_{p} \times \mathbb{Z}_{p(p-1)}) \rtimes \mathbb{Z}_{p(p-1)}$.

\vspace{0.2 cm}

\n \textbf{The group $G_4$}. Let $H = \langle a,b \;|\; a^{p^{2}} = b^{p} = 1, ab = ba\rangle$ and $K=\langle c \;|\; c^{p} = 1\rangle$. Then	$G_4 \simeq H\rtimes_{\phi} K$, where $\phi: K \longrightarrow Aut(H)$ is given by $\phi(c)(a) = a^{1+p}$ and $\phi(c)(b) = b$. Note that  $[H,K] = \langle a^{p}\rangle \simeq \mathbb{Z}_{p}$. By the similar argument as in the case $G_3$ above, $C_{K}(H) = \{1\}$. Since $K$ is abelian, by \cite[Corollary 2.2, p. 490]{bdw2006} $C$ is the trivial group. Note that, any $\alpha\in Aut(H)$ defined by, $\alpha(a) = a^{i}b^{j}$ and $\alpha(b) = a^{pm}b^{l}$  can be expressed as a matrix $\begin{pmatrix}
i & j\\
m & l
\end{pmatrix}$, where $0\leq i \leq p^2-1, \gcd(p,i) = 1$, $0\le m,j \le p-1$ and $1\leq l \leq p-1$. Also, let $\delta\in Aut(K)\simeq \mathbb{Z}_{p-1}$ be defined by $\delta(c) = c^{r}$, where $1\le r\le p-1$. Now, if $(\alpha, \delta)\in W$, then $(i)$ $\alpha(a^{c}) = \alpha(a)^{\delta(c)}$ and $(ii)$ $\alpha(b^{c}) = \alpha(b)^{\delta(c)}$. Note that $\alpha(b^{c})=\alpha(b)=a^{pm}b^{l}$ and $\alpha(b)^{\delta(c)} = (a^{pm}b^{l})^{c^{r}} = (a^{pm})^{c^{r}}b^{l} = a^{pm(1+p)^{r}}b^{l} = a^{pm}b^{l}$. Therefore, each $\alpha\in Aut(H)$ satisfies $(ii)$.  Now, by $(i)$, $(a^{i}b^{j})^{1+p} = \alpha(a^{1+p}) = \alpha(a^{c}) = \alpha(a)^{\delta(c)} = (a^{i}b^{j})^{c^{r}} = (a^{i})^{c^{r}}b^{j} = a^{i(1+p)^{r}}b^{j}$. Thus, $i(p+1) \equiv i(1+p)^{r} \;(\mod\;p^{2})$ which implies that $r = 1$. Therefore, $W \simeq Aut(H)\simeq  \mathbb{Z}_{p-1} \times (((\mathbb{Z}_{p} \times \mathbb{Z}_{p})\rtimes\mathbb{Z}_{p})\rtimes \mathbb{Z}_{p-1})$. Hence, $E\simeq \mathbb{Z}_{p-1} \times (((\mathbb{Z}_{p} \times \mathbb{Z}_{p})\rtimes\mathbb{Z}_{p})\rtimes \mathbb{Z}_{p-1})$. Thus, $Aut_{K}(G_4) \simeq  \mathbb{Z}_{p-1} \times (((\mathbb{Z}_{p} \times \mathbb{Z}_{p})\rtimes\mathbb{Z}_{p})\rtimes \mathbb{Z}_{p-1})$.

\vspace{0.2 cm}

\n \textbf{The group $G_5$}. Let $H = \langle b,c \;|\; b^{p} = c^{p} = 1, bc = cb\rangle$ and $K=\langle a \;|\; a^{p^{2}} = 1\rangle$. Then $G_5 \simeq H\rtimes_{\phi} K$, where $\phi: K \longrightarrow Aut(H)$ is given by $\phi(a)(b) = b$ and $\phi(a)(c) = b^{-1}c$.  Note that $[H,K] = \langle b\rangle \simeq \mathbb{Z}_{p}$. Also, if $a^{s}\in C_{K}(H)$, then $b^{i}c^{j} = a^{s}b^{i}c^{j}a^{-s} = b^{i-js}c^{j}$. Therefore, $s\equiv 0 \;(\mod\;p)$ and $C_{K}(H) = \langle a^{p}\rangle$. Since $K$ is abelian, by \cite[Corollary 2.2, p. 490]{bdw2006}, $\gamma(h^k)=\gamma(h)$ is equivalent to $\gamma \in Hom(H/[H,K], K)$. Define $\gamma_{k}\in Hom(H/[H,K],K)$ by $\gamma_{k}(b) = 1$ and $\gamma_{k}(c) = a^{pk}$ for all $0\le k \le p-1$. Since $[H,K]\subseteq Ker \gamma_k$, it will induce a homomorphism from $H/[H,K]$ to $K$. Let $c = \begin{pmatrix}
1 & 0\\
\gamma_{1} & 1
\end{pmatrix}$. Then, one can easily observe that $\gamma_{1}(H)\subseteq C_{K}(H)$. Therefore, $Hom(H/[H,K], K) \simeq C = \langle c\rangle \simeq \mathbb{Z}_{p}$. Note that, any $\alpha\in Aut(H)\simeq GL(2, p)$ defined as, $\alpha(b) = b^{i}c^{j}$ and $\alpha(c) = b^{l}c^{m}$  can be represented as a matrix, $\begin{pmatrix}
i & j\\
l & m
\end{pmatrix}$, where $0\le l,j \le p-1$ and $1 \le i,m \le p-1$. Also, let $\delta\in Aut(K)\simeq \mathbb{Z}_{p(p-1)}$ be defined by $\delta(a) = a^{r}$, where $r\in \mathbb{Z}_{p^{2}}, \gcd(p,r) = 1$. Now, if $(\alpha, \delta)\in W$, then $(i)$ $ \alpha(b^{a}) = \alpha(b)^{\delta(a)}$ and $(ii)$ $\alpha(c^{a}) = \alpha(c)^{\delta(a)}$. By $(i)$, $b^{i}c^{j} = \alpha(b) = \alpha(b^{a}) = \alpha(b)^{\delta(a)} = (b^{i}c^{j})^{a^{r}} = b^{i}b^{-rj}c^{j} = b^{i - rj}c^{j}$. Thus, $rj \equiv 0 \;(\mod\;p)$ which implies that $j = 0$. Now, by $(ii)$, $b^{-i+l}c^{m} = \alpha(b^{-1}c) = \alpha(c^{a}) = \alpha(c)^{\delta(a)} = (b^{l}c^{m})^{a^{r}} = b^{l}(c^{m})^{a^{r}} = b^{l}b^{-rm}c^{m} = b^{l-rm}c^{m}$. Thus, $i \equiv rm \;(\mod\; p)$. Let $t$ be a primitive root of $1 \;(\mod\; p)$ and $x = \left(\begin{pmatrix}
t & 0\\
1 & t
\end{pmatrix}, \delta_{1}\right)$, and $y = \left(\begin{pmatrix}
t+p & 0\\
0 & 1
\end{pmatrix}, \delta_{t}\right)$, where $\delta_{\rho}(a) = a^{\rho}$. Then $W \simeq \langle x,y \;|\; x^{p(p-1)} =1, y^{p(p-1)} = 1, yxy^{-1} = x^{\lambda}\rangle$, where $x^{\lambda} = \begin{pmatrix}
t & 0\\
(t+p)^{-1} & t
\end{pmatrix}$. Then $W\simeq \mathbb{Z}_{p(p-1)}\rtimes \mathbb{Z}_{p(p-1)}$ and so, $E \simeq \mathbb{Z}_{p(p-1)}\rtimes \mathbb{Z}_{p(p-1)}$. Hence, $Aut_{K}(G_5) \simeq \mathbb{Z}_{p} \times (\mathbb{Z}_{p(p-1)}\rtimes \mathbb{Z}_{p(p-1)})$.

\vspace{0.2 cm}

\n \textbf{The group $G_6$}. Let $H = \langle a,b \;|\; a^{p^{2}} = b^{p} = 1, ba = a^{1+p}b\rangle$ and $K=\langle c \;|\; c^{p} = 1\rangle$. Then $G_6 \simeq H\rtimes_{\phi} K$, where $\phi: K \longrightarrow Aut(H)$ is given by $\phi(c)(a) = ab$ and $\phi(c)(b) = b$. Note that $[H,K] = \langle b^{-1},a^{p}\rangle \simeq \mathbb{Z}_{p}\times \mathbb{Z}_{p}$. By the similar argument as in the case $G_3$ above, $C_{K}(H) = \{1\}$ and hence $C$ is the trivial group. Now, $\alpha \in Aut(H)$ as given in \cite{bdw2006} can be expressed as a matrix $\begin{pmatrix}
\eta & \beta\\
\xi & 1
\end{pmatrix}$, where $\eta(a) = a^{i}, 0\leq i\leq p^2-1, \gcd(p,i) = 1$, $\beta(b) = a^{pj}, 0\le j \le p-1$, $\xi(a) = b^{k}$, $0\le k \le p-1$, and $1(b) = b$. Also, $\delta \in Aut(K)$ is given by $\delta(c) = c^{r}, 1\le r \le p-1$. Now, if $(\alpha, \delta)\in W$, then $(i)$ $ \alpha(a^{c}) = \alpha(a)^{\delta(c)}$ and $(ii)$ $\alpha(b^{c}) = \alpha(b)^{\delta(c)}$. Note that $\alpha(b^{c})=\alpha(b)=a^{pj}b$ and $\alpha(b)^{\delta(c)} = (a^{pj}b)^{c^{r}} = (c^{r}ac^{-r})^{pj}b = (ab^{r})^{pj}b = a^{pj+rp\frac{pj(pj-1)}{2}}b^{pjr+1} = a^{pj}b$. Therefore, each $\alpha \in Aut(H)$ satisfies $(ii)$.  Now, by $(i)$, $a^{i+pj}b^{k+1} = \alpha(ab) = \alpha(a^{c}) = \alpha(a)^{\delta(c)} = (a^{i}b^{k})^{c^{r}} = (c^{r}ac^{-r})^{i}b^{k} = (ab^{r})^{i}b^{k} = a^{i+rp\frac{i(i-1)}{2}}b^{ri+k}$. Thus, $ri \equiv 1 \;(\mod\; p)$ which gives that $i \equiv 2j + 1 \;(\mod\; p)$. Therefore, $i\in \{(2j+1)+\lambda p \;|\; \lambda \in \mathbb{Z}_{p}\}$. Let $t$ be a primitive root of $1 \;(\mod\; p)$ and $x = \left(\begin{pmatrix}
t+p & 0\\
0 & 1
\end{pmatrix}, \delta_{t}\right), y = \left(\begin{pmatrix}
	1 & 0\\
	1 & 1
\end{pmatrix}, \delta_{1}\right)$, and $z =\left(\begin{pmatrix}
1+p & 0\\
0 & 1
\end{pmatrix}, \delta_{1}\right)$, where $\delta_{\rho}(a) = a^{\rho}$. Then $W \simeq \langle x,y,z \;|\; x^{p(p-1)} = 1 = y^{p}= z^{p}, xyx^{-1} = y^{e}, xz = zx, yz = zy\rangle$, where $y^{e} = \begin{pmatrix}
1 & 0\\
(t+p)^{-1} & 1
\end{pmatrix}$. Hence, $W \simeq \mathbb{Z}_{p}\times ((\mathbb{Z}_{p}\times \mathbb{Z}_{p})\rtimes \mathbb{Z}_{p-1})$ and so, $E \simeq \mathbb{Z}_{p}\times ((\mathbb{Z}_{p}\times \mathbb{Z}_{p})\rtimes \mathbb{Z}_{p-1})$. Thus, $Aut_{K}(G_6) \simeq  \mathbb{Z}_{p}\times ((\mathbb{Z}_{p}\times \mathbb{Z}_{p})\rtimes \mathbb{Z}_{p-1})$.

\vspace{0.2 cm}

\n \textbf{The group $G_7$}. Let $H = \langle a,b \;|\; a^{p^{2}} = b^{p} = 1, ba = a^{1+p}b\rangle$ and $K=\langle c \;|\; c^{p} = 1\rangle$. Then, $G_7 \simeq H\rtimes_{\phi} K$, where $\phi: K \longrightarrow Aut(H)$ is given by $\phi(c)(a) = a^{1+p}b$ and $\phi(c)(b) = a^{p}b$. Note that $[H,K] = \langle b,a^{p}\rangle \simeq \mathbb{Z}_{p}\times \mathbb{Z}_{p}$. By the similar argument as in the case $G_3$ above, $C_{K}(H) = \{1\}$ and hence $C$ is the trivial group. Each $\alpha \in Aut(H)$ can be expressed as a matrix $\begin{pmatrix}
\eta & \beta\\
\xi & 1
\end{pmatrix}$, where $\eta(a) = a^{i}, 0\leq i\leq p^2-1, \gcd(i,p)=1$, $\beta(b) = a^{pj}, \xi(a) = b^{k}, 0\le j,k \le p-1$, and $1(b) = b$. Also, $\delta \in Aut(K)$ is given by $\delta(c) = c^{r}, 1\le r \le p-1$. Now, if $(\alpha, \delta)\in W$, then $(i)$ $ \alpha(a^{c}) = \alpha(a)^{\delta(c)}$ and $(ii) \alpha(b^{c}) = \alpha(b)^{\delta(c)}$.

\vspace{.2cm}

\n By $(ii)$, $a^{pi + pj}b = \alpha(a^{p}b) = \alpha(b^{c}) = \alpha(b)^{\delta(c)} = (a^{pj}b)^{c^{r}} = (c^{r}ac^{-r})^{pj}(c^{r}bc^{-r}) = (a^{1+p\frac{r(r+1)}{2}}b^{r})^{pj}(a^{rp}b) = a^{p^{2}j\frac{r(r+1)}{2}}(ab^{r})^{pj}a^{rp}b = a^{pj+pr}b$. Thus $i \equiv r \;(\mod\; p)$. Now, by $(i)$, $a^{i(1+p)+pj}b^{k+1} = \alpha(a^{1+p}b) = \alpha(a^{c}) = \alpha(a)^{\delta(c)} = (a^{i}b^{k})^{c^{r}} = (c^{r}ac^{-r})^{i}(c^{r}bc^{-r})^{k} = a^{i+pri\frac{r+i}{2}}b^{ri}(a^{rp}b)^{k} = a^{i+rpi(\frac{r+i}{2})+rpk}b^{ri+k}$. Thus, $ri \equiv 1 \;(\mod\; p)$ and $ip +pj \equiv pri(\frac{r+i}{2})+rpk \;(\mod\; p^{2})$ implies that $i+j\equiv r+rk \;(\mod\; p)$. So, $j \equiv rk \;(\mod\; p)$. Using $r\equiv i \;(\mod\; p)$ and $ri \equiv 1 \;(\mod\; p)$, we get $i^{2}\equiv 1 \;(\mod\; p^{2})$. Let $t$ be a primitive root of $1 \;(\mod\; p)$ and $x = \left(\begin{pmatrix}
1 & 1\\
0 & 1
\end{pmatrix}, \delta_{1}\right), y = \left(\begin{pmatrix}
-1 & 0\\
 0 & 1
\end{pmatrix}, \delta_{1}\right)$ and $z= \left(\begin{pmatrix}
t+p & 0\\
0 & 1
\end{pmatrix}, \delta_{t}\right)$, where $\delta_{\rho}(a) = a^{\rho}$. Then $W \simeq \langle x,y,z \;|\; x^{p} = 1 = y^{2} = z^{p(p-1)}, xy = yx^{-1}, zxz^{-1} = x^{t}, yz=zy\rangle \simeq \mathbb{Z}_{p(p-1)}\times (\mathbb{Z}_{p}\rtimes \mathbb{Z}_{2})$ and so, $E \simeq \mathbb{Z}_{p(p-1)}\times (\mathbb{Z}_{p}\rtimes \mathbb{Z}_{2})$. Hence, $Aut_{K}(G_7) \simeq  \mathbb{Z}_{p(p-1)}\times (\mathbb{Z}_{p}\rtimes \mathbb{Z}_{2}) \simeq  D_{2p}\times \mathbb{Z}_{p(p-1)}$, where $D_{2p}$ is the dihedral group of order $2p$. 

\vspace{0.2 cm}

\n \textbf{The group $G_8$}. Let $H = \langle a,b \;|\; a^{p^{2}} = b^{p} = 1, ba = a^{1+p}b\rangle$ and $K=\langle c \;|\; c^{p} = 1\rangle$. Then $G_8 \simeq H\rtimes_{\phi} K$, where $\phi: K \longrightarrow Aut(H)$ is given by $\phi(c)(a) = a^{1+dp}b$ and $\phi(c)(b) = a^{dp}b$, $d \not\equiv 0,1 \;(\mod\; p)$. By the similar argument as for the group $G_{7}$, we get, $C$ is the trivial group and $E \simeq \mathbb{Z}_{p(p-1)}\times (\mathbb{Z}_{p}\rtimes \mathbb{Z}_{2})$. Hence, $Aut_{K}(G_8) \simeq D_{2p}\times \mathbb{Z}_{p(p-1)}$.

\vspace{0.2 cm}

\n \textbf{The group $G_9$}. Let $H =\langle a,b,c \;|\; a^{p} = b^{p}= c^{p} = 1, ab = ba, bc = cb, ac = ca\rangle$, and $K = \langle d \;|\; d^{p} = 1 \rangle$. Then $G_9 \simeq H\rtimes_\phi K$, where $\phi: K \longrightarrow Aut(H)$ is given by $\phi(d)(a) = a, \phi(d)(b) = b$, and $\phi(d)(c) = ac$. 

\vspace{.2cm}

\n Note that $[H,K] = \langle a \rangle \simeq \mathbb{Z}_{p}$. By the similar argument as in the case $G_3$ above, $C_{K}(H) = \{1\}$ and hence $C$ is the trivial group. Note that, $Aut(H) \simeq GL(3,p)$ and $Aut(K)\simeq \mathbb{Z}_{p-1}$. So, any automorphism $\alpha \in Aut(H)$ can be identified as an element in $GL(3,p)$. Let $\alpha\in Aut(H)$ and $\delta \in Aut(K)$ be defined as, $\alpha(a) = a^{i}b^{j}c^{k}, \alpha(b) = a^{l}b^{m}c^{n}, \alpha(c) = a^{\lambda}b^{\mu}c^{\rho}$, and $\delta(d) = d^{r}$, where $1\le i,m,\rho,r \le p-1$ and $0\le j,k,l,n,\lambda, \mu \le p-1$. Now, if $(\alpha, \delta)\in W$, then $(i)$ $\alpha(a^{d}) = \alpha(a)^{\delta(d)}$, $(ii)$ $\alpha(b^{d}) = \alpha(b)^{\delta(d)}$ and $(iii)$ $\alpha(c^{d}) = \alpha(c)^{\delta(d)}$. 

\vspace{.2cm}

\n Note that, $d^{r}cd^{-r} = a^{r}c$. By $(i)$, $a^{i}b^{j}c^{k} = \alpha(a) = \alpha(a^{d}) = \alpha(a)^{\delta(d)} = (a^{i}b^{j}c^{k})^{d^{r}} = a^{i}b^{j}(d^{r}cd^{-r})^{k} = a^{i}b^{j}(a^{r}c)^{k} = a^{i+rk}b^{j}c^{k}$. Therefore, $rk \equiv 0 (\mod p)$ which implies that $k = 0$. Now, by $(ii)$, $a^{l}b^{m}c^{n} = \alpha(b) = \alpha(b^{d}) = \alpha(b)^{\delta(d)} = (a^{l}b^{m}c^{n})^{d^{r}} = a^{l+rn}b^{m}c^{n}$. Therefore, $rn \equiv 0 (\mod p)$ which implies that $n = 0$. By $(iii)$, $a^{i+\lambda}b^{j+\mu}c^{\rho} = \alpha(ac) = \alpha(c^{d}) = \alpha(c)^{\delta(d)} = (a^{\lambda}b^{\mu}c^{\rho})^{d^{r}} = a^{\lambda+r\rho}b^{\mu}c^{\rho}$. Thus, $i = r\rho$ and $j = 0$. So, we have, $\alpha = \begin{pmatrix}
r\rho & 0 & 0\\
l & m & 0\\
\lambda & \mu & \rho
\end{pmatrix}$. Let $t$ be a primitive root of $1 \;(\mod\; p)$ and $u = \left(\begin{pmatrix}
t & 0 & 0\\
0 & 1 & 0\\
0 & 0 & t
\end{pmatrix}, \delta_{1}\right), v = \left(\begin{pmatrix}
1 & 0 & 0\\
0 & t & 0\\
0 & 0 & 1
\end{pmatrix}, \delta_{1}\right), w = \left(\begin{pmatrix}
t & 0 & 0\\
0 & 1 & 0\\
0 & 0 & 1
\end{pmatrix}, \delta_{t}\right), x= \left(\begin{pmatrix}
1 & 0 & 0\\
1 & 1 & 0\\
0 & 0 & 1
\end{pmatrix}, \delta_{1}\right), y = \left(\begin{pmatrix}
1 & 0 & 0\\
0 & 1 & 0\\
1 & 0 & 1
\end{pmatrix}, \delta_{1}\right), z = \left(\begin{pmatrix}
1 & 0 & 0\\
0 & 1 & 0\\
0 & 1 & 1
\end{pmatrix}, \delta_{1}\right)$, where $\delta_{s}(d) = d^{s}$. Then $W \simeq \langle u,v,w,x,y,z \; |\; u^{p-1} = 1 = v^{p-1}= w^{p-1} = x^{p} = y^{p} = z^{p}, uv=vu, uw=wu, uy=yu, vw=wv, vy=yv, wz=zw, xy=yx, yz=zy, uxu^{-1} = x^{t^{-1}}, uzu^{-1} = z^{t}, vxv^{-1} = x^{t}, vzv^{-1} = z^{t^{-1}}, wxw^{-1} = x^{t^{-1}}, wyw^{-1} = y^{t^{-1}}, zx = xyz\rangle \simeq (((\mathbb{Z}_{p}\times \mathbb{Z}_{p})\rtimes \mathbb{Z}_{p})\times \mathbb{Z}_{p-1})\rtimes (\mathbb{Z}_{p-1}\times \mathbb{Z}_{p-1})$ and so, $E \simeq (((\mathbb{Z}_{p}\times \mathbb{Z}_{p})\rtimes \mathbb{Z}_{p})\times \mathbb{Z}_{p-1})\rtimes (\mathbb{Z}_{p-1}\times \mathbb{Z}_{p-1})$. Hence, $Aut_{K}(G_9) \simeq (((\mathbb{Z}_{p}\times \mathbb{Z}_{p})\rtimes \mathbb{Z}_{p})\times \mathbb{Z}_{p-1})\rtimes (\mathbb{Z}_{p-1}\times \mathbb{Z}_{p-1})$.

\vspace{0.2 cm}

\n \textbf{The group $G_{10}$}.	Let $H =\langle a,b,c \;|\; a^{p} = b^{p}= c^{p} = 1, ab = ba, bc = cb, ac = ca\rangle$ and $K = \langle d \;|\; d^{p} = 1 \rangle$. Then, $G_{10} \simeq H\rtimes_\phi K$, where $\phi: K \longrightarrow Aut(H)$ is given by $\phi(d)(a) = a, \phi(d)(b) = ab$, and $\phi(d)(c) = bc$. 

\vspace{.2cm}

\n Note that $[H,K] = \langle a,b \rangle \simeq \mathbb{Z}_{p}\times \mathbb{Z}_{p}$. By the similar argument as above $C$ is the trivial group. Note that, $Aut(H) \simeq GL(3,p)$ and $Aut(K)\simeq \mathbb{Z}_{p-1}$. So, any automorphism $\alpha \in Aut(H)$ can be identified as an element in $GL(3,p)$. Let $\alpha\in Aut(H)$ and $\delta \in Aut(K)$ be defined as, $\alpha(a) = a^{i}b^{j}c^{k}, \alpha(b) = a^{l}b^{m}c^{n}, \alpha(c) = a^{\lambda}b^{\mu}c^{\rho}$, and $\delta(d) = d^{r}$, where $1\le i,m,\rho,r \le p-1$ and $0\le j,k,l,n,\lambda, \mu \le p-1$. Now, if $(\alpha, \delta)\in W$, then $(i)$ $\alpha(a^{d}) = \alpha(a)^{\delta(d)}$, $(ii)$ $\alpha(b^{d}) = \alpha(b)^{\delta(d)}$ and $(iii)$ $\alpha(c^{d}) = \alpha(c)^{\delta(d)}$. 

\vspace{.2cm}

\n Note that, $d^{r}bd^{-r} = a^{r}b$ and $d^{r}cd^{-r} = a^{\frac{r(r-1)}{2}}b^{r}c$. By $(i)$, $a^{i}b^{j}c^{k} = \alpha(a) = \alpha(a^{d}) = \alpha(a)^{\delta(d)} = (a^{i}b^{j}c^{k})^{d^{r}} = a^{i}(d^{r}bd^{-r})^{j}(d^{r}cd^{-r})^{k} = a^{i}(a^{r}b)^{j}\\(a^{\frac{r(r-1)}{2}}b^{r}c)^{k} = a^{i+rj+k\frac{r(r-1)}{2}}b^{j+rk}c^{k}$. Thus $k = 0$ and $j = 0$. Now, by $(ii)$, $a^{i+l}b^{m}c^{n} = \alpha(ab) = \alpha(b^{d}) = \alpha(b)^{\delta(d)} = (a^{l}b^{m}c^{n})^{d^{r}} = a^{l+rm+n\frac{r(r-1)}{2}}b^{m+rn}c^{n}$. Thus, $n = 0$ and $i = rm$. By $(iii)$, $a^{l+\lambda}b^{m+\mu}c^{\rho} = \alpha(bc) = \alpha(c^{d}) = \alpha(c)^{\delta(d)} = (a^{\lambda}b^{\mu}c^{\rho})^{d^{r}} = a^{\lambda+r\mu+\rho\frac{r(r-1)}{2}}b^{\mu+r\rho}c^{\rho}$. Thus, $m = r\rho$ and $l = r\mu+\rho\frac{r(r-1)}{2} \;(\mod\; p)$. So, we have, $\alpha = \begin{pmatrix}
r^{2}\rho & 0 & 0\\
l & r\rho & 0\\
\lambda & \mu & \rho
\end{pmatrix}$, where $l = r\mu+\rho\frac{r(r-1)}{2} \;(\mod\; p)$. Let $t$ be a primitive root of $1 \;(\mod\; p)$ and $x = \left(\begin{pmatrix}
t & 0 & 0\\
0 & t & 0\\
0 & 0 & t
\end{pmatrix}, \delta_{1}\right), y = \left(\begin{pmatrix}
t^{2} & 0 & 0\\
0 & t & 0\\
0 & 0 & 1
\end{pmatrix}, \delta_{t}\right)$ and $z = \left(\begin{pmatrix}
1 & 0 & 0\\
1 & 1 & 0\\
1 & 1 & 1
\end{pmatrix}, \delta_{1}\right)$, where $\delta_{s}(d) = d^{s}$. Note that, $\langle z \rangle = \langle \begin{pmatrix}
1 & 0 & 0\\
1 & 1 & 0\\
 0& 1 & 1
\end{pmatrix}, \begin{pmatrix}
1 & 0 & 0\\
0 & 1 & 0\\
1 & 0 & 1
\end{pmatrix}\rangle$ is an abelian group of order $p^{2}$. Therefore, $W \simeq \langle x,y,z \;|\; x^{p-1} = y^{p-1} = z^{p}, xy= yx, xz=zx, yzy^{-1} = z^{u}\rangle$, where $z^{u} = \begin{pmatrix}
1 & 0 & 0\\
t^{-1} & 1 & 0\\
t^{-2} & t^{-1} & 1
\end{pmatrix}$. Thus $W\simeq (\mathbb{Z}_{p}\times \mathbb{Z}_{p})\rtimes (\mathbb{Z}_{p-1}\times \mathbb{Z}_{p-1})$ and so, $E \simeq (\mathbb{Z}_{p}\times \mathbb{Z}_{p})\rtimes (\mathbb{Z}_{p-1}\times \mathbb{Z}_{p-1})$.  Hence, $Aut_{K}(G_{10}) \simeq \mathbb{Z}_{p-1}\times ((\mathbb{Z}_{p}\times \mathbb{Z}_{p})\rtimes \mathbb{Z}_{p-1})$.

	\section*{Acknowledgements}
	The first author is supported by the Junior Research Fellowship of UGC, India.
	
\end{document}